\documentclass{amsart}
\usepackage{graphicx}
\usepackage{amssymb,amsthm,amstext,amsmath,amscd,latexsym,amsfonts,amscd,enumerate}
\usepackage{epstopdf}
\usepackage[all,cmtip]{xy}
\usepackage[hidelinks=true]{hyperref}
\usepackage{stmaryrd} 
\usepackage{harpoon} 

\usepackage{mathrsfs}
\usepackage{mathrsfs}

\newcommand{\fa}{\frak{a}}
\newcommand{\fb}{\frak{b}}
\newcommand{\fc}{\frak{c}}
\newcommand{\fm}{\frak{m}}

\newcommand{\fq}{\frak{q}}

\newcommand{\QQ}{\mathbb{Q}}

\def\im{\operatorname{im}}

\def\ker{\operatorname{ker}}

\def\dirlim{\varinjlim}

\def\fm{\mathfrak m}

\def\Hom{\operatorname{Hom}}

\def\Ext{\operatorname{Ext}}

\newcommand{\Soc}{\operatorname{Soc}}
\newcommand{\ann}{\operatorname{ann}}

\numberwithin{equation}{section}

\theoremstyle{plain}
\newtheorem{thm}[equation]{Theorem}          \newtheorem*{thm*}{Theorem}
\newtheorem{prp}[equation]{Proposition}      \newtheorem*{prp*}{Proposition}
\newtheorem{cor}[equation]{Corollary}        \newtheorem*{cor*}{Corollary}
\newtheorem{lem}[equation]{Lemma}            \newtheorem*{lem*}{Lemma}
          \newtheorem*{cnj*}{Conjecture}
            \newtheorem*{fct*}{Fact}

\theoremstyle{definition}
\newtheorem{dfn}[equation]{Definition}       \newtheorem*{dfn*}{Definition}
     \newtheorem*{con*}{Construction}
     \newtheorem*{funcon*}{Functorial Constructions}
      \newtheorem*{obs*}{Observation}
\newtheorem{rmk}[equation]{Remark}           \newtheorem*{rmk*}{Remark}
\newtheorem{exa}[equation]{Example}          \newtheorem*{exa*}{Example}
         \newtheorem*{exe*}{Exercise}
         \newtheorem*{qst*}{Question}
            \newtheorem{stp*}{Setup}
            \newtheorem*{set*}{Setting}
            \newtheorem*{ntn*}{Notation}

\theoremstyle{plain}

\usepackage{xcolor}

\newcommand{\vecu}{\overrightharp{u}}
\newcommand{\vecv}{\overrightharp{v}}
\newcommand{\vecw}{\overrightharp{w}}

\newcommand{\vecz}{\overrightharp{z}}
\newcommand{\vecr}{\overrightharp{r}}

\title[Reducibility of parameter ideals]{Reducibility of parameter ideals in low powers of the maximal ideal} 

\author{Katharine Shultis}
\address[K. Shultis]{Department of Mathematics, Gonzaga University, Spokane, WA 99258, USA}
\email{shultis@gonzaga.edu}
\author{Peder Thompson} 
\address[P. Thompson]{Institutt for matematiske fag, Norwegian University of Science and Technology, N-7491 Trondheim, Norway}
\email{peder.thompson@ntnu.no}

\date{June 8, 2020} 

\makeatletter
\@namedef{subjclassname@2020}{%
  \textup{2020} Mathematics Subject Classification}
\makeatother

\subjclass[2020]{13C05, 13D45, 13H10}
\keywords{System of parameters, reducible parameter ideal, Gorenstein ring}

\thanks{This material is based upon work supported by the National Science Foundation under Grant \#1321794, as part of the Mathematical Research Communities 2015 program in Snowbird, Utah.}

\dedicatory{To Roger and Sylvia Wiegand in celebration of their combined 150th birthday.}

\begin{document}
\maketitle

\begin{abstract}
A commutative noetherian local ring $(R,\fm)$ is Gorenstein if and only if every parameter ideal of $R$ is irreducible. Although irreducible parameter ideals may exist in non-Gorenstein rings, Marley, Rogers, and Sakurai show there exists an integer $\ell$ (depending on $R$) such that $R$ is Gorenstein if and only if there exists an irreducible parameter ideal contained in $\fm^\ell$. We give upper bounds for $\ell$ that depend primarily on the existence of certain systems of parameters in low powers of the maximal ideal.
\end{abstract}

\section{Introduction}
\noindent
Let $(R,\fm,k)$ be a commutative noetherian local ring of dimension $\dim R=d$.  It is known that $R$ is Gorenstein if and only if every parameter ideal of $R$ is irreducible, but we cannot characterize Gorenstein rings by the existence of an irreducible parameter ideal. For example, the non-Gorenstein ring $\QQ\llbracket x,y \rrbracket/(x^2,xy)$ has an irreducible parameter ideal $(y)$, although $(y^j)$ is reducible for $j\geq 2$. Marley, Rogers, and Sakurai show \cite{MRS08}, however, that the existence of a parameter ideal in a sufficiently high power of the maximal ideal {\em does} characterize Gorenstein rings:

\begin{thm}[Marley, Rogers, and Sakurai]\label{MRS08}
There exists an integer $\ell$, depending on $R$, such that $R$ is Gorenstein if and only if some parameter ideal contained in $\fm^\ell$ is irreducible.
\end{thm}

The integer $\ell$ in Theorem \ref{MRS08}, considered previously by Goto and Sakurai \cite[Lemma 3.12]{GS03} may be taken to be the least integer $i$ such that the canonical map
\begin{align}\label{can_map}
\xymatrix{\Ext_R^d(R/\fm^i,R)\ar[r]& \dirlim_j\Ext_R^d(R/\fm^j,R)\cong H_\fm^d(R)}
\end{align}
becomes surjective after applying the socle functor $\Hom_R(R/\fm,-)$. The existence of such an integer is guaranteed as the socle module $\Hom_R(R/\fm,H_\fm^d(R))$ is finitely generated, but determining how large $\ell$ must be seems to be somewhat subtle. Indeed, we show in Example \ref{exampledim1} that for each integer $a\geq 1$, there exists a ring which requires $\ell>a$. To understand how deep in the maximal ideal one must go before detecting whether $R$ is Gorenstein in terms of reducibility of parameter ideals, we consider the problem, posed to the authors by Marley, of finding an upper bound for the integer $\ell$ in Theorem \ref{MRS08}.

For rings of dimension one, we take a direct approach to determine surjectivity of the maps 
in \eqref{can_map} after applying $\Hom_R(R/\fm,-)$. We show (see Theorem \ref{main}):
\begin{thm}\label{intro1} 
Assume $\dim R=1$ and $k$ is infinite. If $n$ is the least integer such that $\fm^{n}=(x)\fm^{n-1}$ for some parameter $x$ and $\fm^{n-1}\cap \Gamma_\fm(R)=0$, then for $i\geq n$ the canonical map in \eqref{can_map} becomes surjective after applying $\Hom_R(R/\fm,-)$.
\end{thm} 
\noindent
Thus, in this setting, $R$ is Gorenstein if and only if some parameter ideal contained in $\fm^{n}$ is irreducible; see Corollary \ref{main_cor}. This consequence of Theorem \ref{intro1} can also be deduced from work of Rogers \cite{Rog04} and Marley, Rogers, and Sakurai \cite{MRS08}. The assumption that $k$ is infinite is only needed to ensure the existence of a parameter $x$ such that $\fm^n=(x)\fm^{n-1}$.

For a ring $R$ of dimension $d$ (not necessarily equal to $1$) and a system of parameters $x_1,...,x_d$, we instead consider---in place of \eqref{can_map}---the least integer $i$ such that the canonical map
\begin{align}\label{can_map2}
\xymatrix{R/(x_1^i,...,x_d^i)\ar[r] & \dirlim_j R/(x_1^j,...,x_d^j)\cong H_\fm^d(R)}
\end{align}
becomes surjective after applying $\Hom_R(R/\fm,-)$. We focus on the case where $x_1,...,x_d$ is a $p_s$-standard system of parameters (see Definition \ref{ps_standard}); this is a variant of the $p$-standard systems of parameters considered by Cuong \cite{Cuo95}. These systems of parameters (both $p_s$-standard and $p$-standard) are chosen in a way as to annihilate certain local cohomology modules. We show in Proposition \ref{simple_surjection} that if $x_1,...,x_d$ is a $p_s$-standard system of parameters for some $s\geq 2$, then for $i\geq s$ the canonical map in \eqref{can_map2} becomes surjective after applying $\Hom_R(R/\fm,-)$.

As a consequence of this surjectivity, we obtain a characterization of the Gorenstein property of $R$ in terms of irreducibility of parameter ideals: the integer $\ell$ from Theorem \ref{MRS08} can be taken to be the integer $n$ in the next result. In particular, we prove (as a special case of Theorem \ref{Gor_char_higherdim}):
\begin{thm}\label{intro2}
Assume $R$ has a dualizing complex. If $n$ is an integer such that $\fm^n\subseteq (x_1^{2},...,x_d^{2})$ for a $p_2$-standard system of parameters $x_1,...,x_d$, then $R$ is Gorenstein if and only if some parameter ideal contained in $\fm^n$ is irreducible.
\end{thm}
\noindent
The assumption that $R$ has a dualizing complex is sufficient for the existence of a $p_2$-standard system of parameters; see the discussion before Definition \ref{ps_standard}. Moreover, work of Cuong and Cuong \cite{CC07,CC17} implies that Theorem \ref{intro2} holds if one replaces the assumption that $R$ has a dualizing complex with the assumption that $R$ is a quotient of a Cohen-Macaulay local ring; see Remark \ref{CuongCuong}.
\begin{equation*}
  * \ \ * \ \ *
\end{equation*}
\noindent
Throughout this paper, let $(R,\fm,k)$ be a commutative noetherian local ring.  Let $\dim R=d$ be the Krull dimension of $R$.  We briefly recall a few facts and notation. 

For an $R$-module $M$, submodule $N\subseteq M$, and ideal $\fa\subseteq R$, we consider the submodule $(N:_M\fa)=\{y\in M \mid \fa y\subseteq N\}$ of $M$. If $\fa=(x)$, just write $(N:_Mx)$. The \emph{socle} of $M$ is $\Soc M=(0:_M\fm)\cong\Hom_R(R/\fm,M)$. The \emph{annihilator} of $M$ is $\ann_RM=(0:_RM)$.

For an ideal $\fa$ of $R$, denote by $\Gamma_\fa(-)$ the $\fa$-torsion functor; its right derived functors yield the usual local cohomology functors, denoted $H_\fa^i(-)$ for $i\geq 0$; for additional background on local cohomology, refer to \cite{BS13,BH98,24hours}. 

A \emph{system of parameters of $R$} is a sequence of elements $x_1,...,x_d$ in $R$ such that $\fm^i\subseteq (x_1,...,x_d)$ for some integer $i$. More generally, if $M$ is an $R$-module with $\dim M=t$, then a sequence $x_1,...,x_t$ in $R$ is a \emph{system of parameters of $M$} if $M/(x_1,...,x_t)M$ has finite length. In either case, an element of a system of parameters is called a \emph{parameter} and an ideal generated by a system of parameters is a \emph{parameter ideal}. For additional facts about systems of parameters, refer to \cite{Eis95,Mat89}.

\section{A bound in dimension one}\label{section_dim1}
\noindent
Assume in this section that the ring $(R,\fm,k)$ has an infinite residue field $k$ and $\dim(R)=1$. Moreover, we fix the next two invariants throughout this section; the first is finite because $k$ is infinite\footnote{For the purposes of this paper, the assumption of $k$ being an infinite field may be replaced with the assumption that a parameter $x$ exists so that $\fm^{i+1}=(x)\fm^i$ for some integer $i$.} \cite[Corollary 4.6.10]{BH98}, the second is finite because $\Gamma_\fm(R)$ has finite length \cite[Theorem 7.1.3]{BS13}:
\begin{align}\label{c_and_g}
\begin{aligned}
c&=\inf\{i \mid \text{there exists a parameter $x$ of $R$ such that $\fm^{i+1}=(x)\fm^i$}\};\\
g&=\inf\{i \mid \fm^i\cap \Gamma_\fm(R)=0\}.
\end{aligned}
\end{align}
These invariants have been considered elsewhere; $c$ is the reduction number of $\fm$, and the bound we consider below, $\max\{c,g\}+1$, is used by Rogers \cite[Theorem 2.3]{Rog04}. We begin with two elementary lemmas involving these invariants:
\begin{lem}\label{parameter_lemma}
Let $x\in R$ be a parameter and $y\in \fm^g$. If $x^iy=0$ for some $i\geq 1$, then $y=0$.
\end{lem}
\begin{proof}
As $(x)$ is a parameter ideal, there exists an integer $j$ such that $\fm^j\subseteq (x)$ hence $\fm^{ij}\subseteq (x^i)$ for $i\geq 1$. If $x^iy=0$, then $\fm^{ij}y=0$. It thus follows that $y\in \fm^g\cap  (0:_R\fm^{ij})\subseteq \fm^g \cap \Gamma_\fm(R)=0$, hence $y=0$.
\end{proof}

\begin{lem}\label{mingens_lemma}
Let $x\in R$ be a parameter with $\fm^{c+1}=(x)\fm^c$ and set $n=\max\{c,g\}$. If $\{y_1,...,y_e\}$ is a minimal generating set of $\fm^n$, then $\{x^iy_1,...,x^iy_e\}$ is a minimal generating set of $\fm^{n+i}$ for each $i\geq 1$.
\end{lem}
\begin{proof}
Let $\{y_1,...,y_e\}$ be a minimal generating set of $\fm^n$. As $n\geq c$, the equality $(x)\fm^n=\fm^{n+1}$ implies $(x^i)\fm^n=\fm^{n+i}$ by induction. Thus $(x^iy_1,...,x^iy_e)=\fm^{n+i}$. If there exists $r_q\in R$ such that $x^iy_j=\sum_{q\not=j} r_q x^iy_q$ for some $j$, then we have $x^i(y_j-\sum_{q\not=j}r_qy_q)=0$. Since $y_j-\sum_{q\not=j}r_qy_q\in \fm^g$, we have $y_j-\sum_{q\not=j}r_qy_q=0$ by Lemma \ref{parameter_lemma}; this contradicts the fact that $y_1,...,y_e$ is a minimal generating set, hence we must have $\{x^iy_1,...,x^iy_e\}$ is a minimal generating set for $\fm^{n+i}$.
\end{proof}

\begin{thm}\label{main}
Assume $k$ is infinite and $\dim(R)=1$.  For $i\geq \max\{c,g\}+1$, the canonical map 
\[\xymatrix{\varphi_i:\Ext_R^1(R/\fm^i,R)\ar[r] &  \dirlim_j \Ext_R^1(R/\fm^j,R)\cong H_\fm^1(R)}\]
becomes surjective after applying $\Soc(-)=\Hom_R(R/\fm,-)$.
\end{thm}
\begin{proof}
Let $x$ be a parameter such that $(x)\fm^c=\fm^{c+1}$ and set $n=\max\{c,g\}$. Let $\{u_1,...,u_e\}$ be a minimal generating set for $\fm^n$.  By Lemma \ref{mingens_lemma}, we know $\{x^{i}u_1,...,x^{i}u_e\}$ is a minimal generating set for $\fm^{n+i}$ for $i\geq 0$. We will consider $\Soc\Ext_R^1(R/\fm^{n+i},R)$ by examining a projective resolution of $R/\fm^{n+i}$.

We first show that, for $i\geq 0$, one may choose free resolutions of $R/\fm^{n+i}$ which agree starting in degree $1$. Set $\vecu=\begin{bmatrix}u_1\, \cdots \, u_e\end{bmatrix}:R^e\to R$. For $i\geq 0$, the containment $\ker(\vecu)\subseteq \ker(x^{i}\vecu)$ is clear, and hence the equality $\ker(\vecu)= \ker(x^{i}\vecu)$ holds because if $\vecr \in \ker(x^i\vecu)$, then $\vecu\vecr=0$ by Lemma \ref{parameter_lemma}. Thus, for $i\geq 0$, there is a matrix $A:R^f\to R^e$ and a commutative diagram with exact rows:
\[
\xymatrix{
R^f\ar[rr]^{A}\ar[d] && R^e \ar[rr]^{x^{i+1}\vecu}\ar[d]^{x} && R \ar[r]\ar[d]^{=} & R/\fm^{n+i+1} \ar[r]\ar@{->>}[d] & 0\\
R^f\ar[rr]^{A} && R^e \ar[rr]^{x^{i}\vecu} && R \ar[r] & R/\fm^{n+i}\ar[r] &0
}
\]
Applying $\Hom_R(-,R)$ to this diagram yields a commutative diagram:
\begin{align}
\begin{aligned}\label{res_dim1}
\xymatrix{
R \ar[rr]^{x^{i+1}\vecu^T} && R^e\ar[rr]^{A^T}&& R^f\\
R \ar[rr]^{x^i\vecu^T}\ar[u]^{=} && R^e\ar[rr]^{A^T}\ar[u]^{x} && R^f\ar[u]
}
\end{aligned}
\end{align}
Taking cohomology, we obtain that $\Ext_R^1(R/\fm^{n+i},R)= \ker(A^T)/\im(x^i\vecu^T)$ for $i\geq 0$. Set $K=\ker(A^T)\subseteq R^e$, $I_i=\im(x^i\vecu^T)\subseteq K$, and identify
$$\Soc\Ext_R^1(R/\fm^{n+i},R)\cong \Soc(K/I_i)$$
for $i\geq 0$. Moreover, for $i\geq 0$, the map $\Ext_R^1(R/\fm^{n+i},R)\to \Ext_R^1(R/\fm^{n+i+1},R)$ is induced by multiplication by $x$, see \eqref{res_dim1}, as well as is the induced map after applying $\Soc(-)$.  Indeed, for $j\geq 1$, the map $\xymatrix{x^j:\Soc(K/I_i)\ar[r] & \Soc(K/I_{i+j})}$ is defined by $\smash{\xymatrix{
\vecz+I_i \ar@{|->}[r] & x^j\vecz+I_{i+j}}}$.

In order to show that $\Soc \varphi_i$ is surjective for $i\geq n+1$, where $\varphi_i$ is as in the statement, it will be enough to show that $\Soc \varphi_{n+1}$ is surjective (this follows from the definition of direct systems). Note that for $\vecv+I_i\in \Soc K/I_i$, the function $\Soc \varphi_i$ is induced by $\varphi_i$ and hence $(\Soc\varphi_i)(\vecv+I_i)=\varphi_i(\vecv+I_i)$.  Hereafter, we use the latter notation.

Let $\sigma\in \dirlim_j\Soc K/I_j$. As $\dirlim_j \Soc K/I_j$ ($\cong \Soc H_\fm^d(R)$) is finitely generated, $\Soc\varphi_i$ is surjective for $i\gg0$, thus $\sigma=\varphi_{n+p}(\vecv+I_{p})$ for some $\vecv+I_p\in \Soc K/I_p$ for some $p\geq 1$. If $p=1$, then $\sigma\in\im\Soc\varphi_{n+1}$ as desired, so assume $p>1$. We proceed by descending induction: that is, we aim to show there is an element $\vecw+I_{p-1}\in \Soc K/I_{p-1}$ such that 
$$x^i(\vecv+I_p)=x^{i+1}(\vecw+I_{p-1})$$
for some $i\geq 1$, and hence $\varphi_{n+p-1}(\vecw+I_{p-1})=\varphi_{n+p}(\vecv+I_p)$.

We consider the element $x^g(\vecv+I_p)=x^g\vecv+I_{p+g}$, recalling that $g$ is the least integer such that $\fm^g\cap \Gamma_\fm(R)=0$. As $x^g\vecv+I_{p+g}$ is a socle element, we have:
\begin{align*}
x(x^g\vecv+I_{p+g})&=0+I_{p+g}\\
\implies & x^{g+1}\vecv\in I_{p+g}\\
\implies & x^{g+1}\vecv=a x^{p+g}\vecu^T\text{ for some $a\in R$, recalling $I_{p+g}=\im(x^{p+g}\vecu^T)$,}\\
\implies & x(x^g\vecv-a x^{p+g-1}\vecu^T)=0\\
\implies & x^g\vecv=a x^{p+g-1}\vecu^T \text{, by Lemma \ref{parameter_lemma}.}
\end{align*}
Since $p\geq 2$, we may set $\vecw=a x^{p-2}\vecu^T$, and notice that $x^g\vecv=x^{g+1}\vecw$. 

We claim $\vecw+I_{p-1}\in \Soc(K/I_{p-1})$. First, $\vecv\in K$ implies that $A^T \vecv=0$, hence $0=x^gA^T\vecv=x^{g+1}A^T\vecw$. As the entries of $\vecw$ are contained in $\fm^g$, so are the entries of $A^T\vecw$.  Lemma \ref{parameter_lemma} yields $A^T\vecw=0$, hence $\vecw\in K$. Next, for any $z\in \fm$, we have
$$zx^{g+1}\vecw=zx^g\vecv=bx^{p+g}\vecu^{T}\text{, for some $b\in R$,}$$
since $x^g\vecv+I_{p+g}$ is a socle element in $K/I_{p+g}$. Thus $x^{g+1}(z\vecw-bx^{p-1} \vecu^T)=0$. The entries of $z\vecw-bx^{p-1} \vecu^T$ are all in $\fm^g$, so Lemma \ref{parameter_lemma} implies that $z\vecw=bx^{p-1}\vecu^T$. Therefore $\vecw+I_{p-1} \in \Soc(K/I_{p-1})$, hence
\begin{align*}
\varphi_{n+p} (\vecv+I_p)&=\varphi_{n+p+g}(x^g\vecv+I_{g+p})\\
&=\varphi_{n+p+g}(x^{g+1}\vecw+I_{g+p})\\
&=\varphi_{n+p-1}(\vecw+I_{p-1}).
\end{align*}
By descending induction, there exists $\vecw'+I_1\in \Soc K/I_1$ such that
$$\varphi_{n+p}(\vecv+I_{p})=\varphi_{n+1}(\vecw'+I_1).$$
The desired map $\Soc\varphi_i$ is therefore surjective for $i\geq n+1=\max\{c,g\}+1$.
\end{proof}

The next example shows that the bound in Theorem \ref{main} is not sharp.
\begin{exa}\label{exa_main_not_sharp}
Let $k$ be an infinite field and and $a\geq 1$ a fixed integer. Consider the local ring $Q=k[\! [x,y]\! ]/(x^{a+1},xy^a)$ with maximal ideal $\fm=(x,y)$. 
A direct computation shows that the invariants in \eqref{c_and_g} for the ring $Q$ satisfy $c\leq a$, since $(y)\fm^a=\fm^{a+1}$, and $g=2a$. Thus $\max\{c,g\}+1=2a+1$.

For $i\geq 2a$, one has $\fm^i=(y^i)$, and the natural surjection $Q/\fm^{i+1}\to Q/\fm^i$ induces a commutative diagram with exact rows:
\[\xymatrix{
\cdots \ar[r] & Q\ar[r]^x\ar[d]^y & Q\ar[r]^{y^{i+1}}\ar[d]^y &  Q\ar[r]\ar[d]^{=} &  Q/\fm^{i+1} \ar[r]\ar[d] & 0\\
\cdots \ar[r] & Q\ar[r]^x & Q\ar[r]^{y^{i}} &  Q\ar[r] &  Q/\fm^i \ar[r] & 0.
}\]
Applying $\Hom_Q(-,Q)$ to this diagram, observe that $\Ext_Q^1(Q/\fm^i,Q)= (x^a,y^a)/(y^i)$ for each $i\geq 2a$. Moreover, the induced maps $\Ext_Q^1(Q/\fm^i,Q) \to \Ext_Q^1(Q/\fm^{i+1},Q)$ are multiplication by $y$.  

Theorem \ref{main} shows the canonical map $\Soc\varphi_i:\Soc\Ext_Q^1(Q/\fm^{i},Q)\to \Soc H_\fm^1(Q)$ is surjective for $i\geq 2a+1$; we claim that in fact $\Soc\varphi_{2a}$ is surjective as well. First note that for $i\geq 2a$, one has $\Soc\Ext_Q^1(Q/\fm^i,Q)\cong (x^ay^{a-1},y^{i-1})/(y^i)$. Thus it will suffice to consider the direct system 
\[\xymatrix{
(x^ay^{a-1},y^{2a-1})/(y^{2a}) \ar[r]^y & (x^ay^{a-1},y^{2a})/(y^{2a+1}) \ar[r]^-y & \cdots.
}\]
In this direct system, the element $x^ay^{a-1}\in(x^ay^{a-1},y^{2a})/(y^{2a+1})$ is sent to $0$, hence surjectivity of $\Soc\varphi_{2a+1}$ implies $\Soc H_\fm^1(Q)$ is generated by $\Soc\varphi_{2a+1}(y^{2a})$. Moreover, $y^{2a}\in (x^ay^{a-1},y^{2a})/(y^{2a+1})$ is the image of $y^{2a-1}\in (x^ay^{a-1},y^{2a-1})/(y^{2a})$, and so $\Soc H_\fm^1(Q)$ is also generated by $\Soc\varphi_{2a}(y^{2a-1})$, thus $\Soc\varphi_{2a}$ is surjective.
\end{exa}

Recall that an ideal $\fq$ of $R$ is reducible if $\fq=\fb\cap \fc$ for two ideals $\fb$ and $\fc$ of $R$ strictly containing $\fq$; if such a decomposition is not possible, then $\fq$ is irreducible.  The following consequence of Theorem \ref{main} allows us to characterize Gorenstein rings in terms of the existence of irreducible parameter ideals in $\fm^n$ for $n=\max\{c,g\}+1$; it can also be obtained using Rogers' \cite[Theorem 2.3]{Rog04} in place of Theorem \ref{main}.

\begin{cor}\label{main_cor}
Assume $k$ is infinite and $\dim(R)=1$. Set $n=\max\{c,g\}+1$. The ring $R$ is Gorenstein if and only if some parameter ideal in $\fm^n$ is irreducible.
\end{cor}
\begin{proof}
This follows from Theorem \ref{main} and \cite[Theorem 2.7]{MRS08}.
\end{proof}

The least integer $\ell$ required to determine whether $R$ is Gorenstein in terms of the existence of an irreducible parameter ideal in $\fm^\ell$ depends on $R$ and is thus at most $\max\{c,g\}+1$ in the case of a dimension $1$ local ring with an infinite residue field. We next show that given an integer $a$, there exists a ring with $a<\ell\leq 2a$.
\begin{exa}\label{exampledim1}
Let us return to the ring $Q=k[\! [x,y]\! ]/(x^{a+1},xy^a)$ and the setting of Example \ref{exa_main_not_sharp}.
The ring $Q$ has dimension 1 and depth 0, hence is non-Gorenstein. As noted above, $\max\{c,g\}+1=2a+1$, so Corollary \ref{main_cor} implies that every parameter ideal in $\fm^{2a+1}$ is reducible, hence $\ell \leq 2a+1$. In fact, the computation in Example \ref{exa_main_not_sharp}---which relies on Theorem \ref{main}---shows that $\ell\leq 2a$ and that every parameter ideal in $\fm^{2a}$ is reducible.

On the other hand, the parameter ideal $(y^a)$ is irreducible, thus $\ell > a$. To see this, it is enough to show that any ideal of $Q$ properly containing $(y^a)$ also contains the nonzero element $x^ay^{a-1}$. Let $\fb\subseteq Q$ be an ideal that properly contains $(y^a)$ and fix $\beta\in \fb\setminus (y^a)$. Write $\beta=\sum_{s,t\geq 0} a_{s,t}x^sy^t$, with $a_{s,t}\in k$. Because $\beta\not\in (y^a)$, the set $\Lambda=\{(s,t) \mid \text{$a_{s,t}\not=0$, $s\leq a$, and $t\leq a-1$}\}$ is nonempty. Choose $(s_0,t_0)\in \Lambda$ with $s_0+t_0\leq s+t$ for all $(s,t)\in \Lambda$. Noting the element $x^{a-s_0}y^{a-1-t_0}\beta$ belongs to $\fm^{2a-1}=(x^ay^{a-1},y^{2a-1})$, it follows that the ideal $\fb$ contains the element $a_{s_0,t_0}^{-1}x^{a-s_0}y^{a-1-t_0}\beta=x^ay^{a-1}+\varepsilon$ with $\varepsilon\in (y^{2a-1})\subseteq (y^a)$. Thus $x^ay^{a-1}\in \fb$.
\end{exa} 

\section{A bound in higher dimensions}\label{section_higherdim}
\noindent
For rings of higher dimension, the problem of determining surjectivity of the socle of the map in \eqref{can_map} becomes more subtle, with obstructions similar to those noted by Fouli and Huneke \cite[Discussion 4.5]{FH11}. In particular, it is not clear to us whether the same type of ``lifting'' technique employed in Theorem \ref{main} can be used to show surjectivity of the socle of the map in \eqref{can_map} if $\dim R>1$. Our solution here is to instead consider surjectivity of the socle of the map in \eqref{can_map2} for $p_s$-standard systems of parameters (defined below).

For this section, $(R,\fm,k)$ is a commutative noetherian local ring with $\dim R=d$. 
\begin{rmk}\label{quotientconnectingmap}
Let $M$ be a finitely generated $R$-module with $\dim M=t>0$ and let $x$ be a parameter of $M$. 
The exact sequence 
$0\to M/(0:_{M}x)\xrightarrow{x} M\to M/xM\to 0$
induces a canonical connecting homomorphism ${H_\fm^{t-1}(M/xM)\to H_\fm^t(M/(0:_{M}x))}$. Moreover, the containment $(x)+\ann_R M\subseteq \ann_R(0:_{M}x)$ implies $\dim (0:_{M}x)<t$, hence the exact sequence $0\to (0:_Mx)\to M\to M/(0:_{M}x)\to 0$ along with Grothendieck's Vanishing Theorem (for example, \cite[6.1.2]{BS13}) yields an isomorphism
\begin{align}\label{modoutcolon}
\xymatrix{H_\fm^t(M) \ar[r]^-{\cong} & H_\fm^t(M/(0:_{M}x)).}
\end{align}
Composing the connecting homomorphism $H_\fm^{t-1}(M/xM)\to H_\fm^t(M/(0:_{M}x))$ from above with the inverse of the isomorphism in \eqref{modoutcolon} yields a homomorphism
\begin{align*}
\xymatrix{
\delta^{M}_{x}:H_\fm^{t-1}(M/xM) \ar[r] & H_\fm^t(M).
}\end{align*} 
\end{rmk}

The next two lemmas are the primary tools for proving one of our main results below, Proposition \ref{simple_surjection}. In light of the isomorphism in \eqref{modoutcolon}, this first lemma essentially follows from a result of Cuong and Quy \cite[Proposition 2.1]{CQ11}, but we spell out some of the details in order to keep track of the map inducing the surjection, which we will need later.
\begin{lem}\label{topcase}
Let $M$ be a finitely generated $R$-module with $\dim M=t>0$. Let $x$ be a parameter of $M$. If $x\in \ann_RH_\fm^{t-1}(M)$, then for $s\geq 2$ the map $\delta^M_{x^s}$ defined in Remark \ref{quotientconnectingmap} induces a split surjection 
\[\xymatrix{
\Soc\delta^M_{x^s}:\Soc H_{\fm}^{t-1}(M/x^sM)\ar[r] & \Soc H_{\fm}^t(M).
}\]
\end{lem}
\begin{proof}
The inclusion $(0:_{M}x)\subseteq (0:_{M}x^s)$ induces the left vertical map in the next commutative diagram with exact rows:
\[\xymatrix{
0 \ar[r] & \frac{{M}}{(0:_{M}x)} \ar@{->>}[d]\ar[r]^-{x} & {M} \ar[r]\ar[d]^{x^{s-1}} & {M}/{x}{M}\ar[r]\ar[d]^{x^{s-1}} & 0\\
0 \ar[r] & \frac{{M}}{(0:_{M}x^s)} \ar[r]^-{x^s} & {M} \ar[r] & {M}/x^s{M}\ar[r] & 0
}\]
From \eqref{modoutcolon}, using that both $x$ and $x^s$ are parameters of $M$, there are isomorphisms $H_{\fm}^t({{M}}/{(0:_{M}x^s)})\cong H_\fm^t({M})\cong H_{\fm}^t({{M}}/{(0:_{M}{x})})$. We thus obtain a commutative diagram with exact rows:
\[\xymatrix{
H_\fm^{t-1}({M}) \ar[r]^{x}\ar[d]  & H_{\fm}^{t-1}({M}) \ar[r]\ar[d]^{x^{s-1}} & H_{\fm}^{t-1}({M}/{x}{M}) \ar[r]^-{\delta'}\ar[d]^{x^{s-1}} & H_{\fm}^t({M}) \ar[r]^{x}\ar[d]^{=} & H_{\fm}^t({M})\ar[d]^{x^{s-1}}\\
H_\fm^{t-1}({M}) \ar[r]^{x^s}  & H_{\fm}^{t-1}({M}) \ar[r] & H_{\fm}^{t-1}({M}/x^s{M}) \ar[r]^-{\delta} & H_{\fm}^t({M}) \ar[r]^{x^s} & H_{\fm}^t({M})
}\]
where $\delta=\delta_{x^s}^M$ and $\delta'=\delta^M_x$ are the maps defined in Remark \ref{quotientconnectingmap}. By assumption, ${x} H_\fm^{t-1}({M})=0$, hence this yields the next commutative diagram with exact rows:
\[\xymatrix{
0 \ar[r] & H_{\fm}^{t-1}({M}) \ar[r]\ar[d]^0 & H_{\fm}^{t-1}({M}/{x}{M}) \ar[r]\ar[d]^{x^{s-1}} & (0:_{H_{\fm}^t({M})}{x}) \ar[r] \ar[d]^{\iota}_{\subseteq} & 0\\
0 \ar[r] & H_{\fm}^{t-1}({M}) \ar[r] & H_{\fm}^{t-1}({M}/x^s{M}) \ar[r]^-{\delta} & (0:_{H_{\fm}^t({M})}x^s) \ar[r] & 0
}\]
Following the argument in \cite[proof of Proposition 2.1]{CQ11}, note that the middle vertical map induces $\varepsilon:(0:_{H_{\fm}^t({M})}{x}) \to H_{\fm}^{t-1}({M}/x^s{M})$ such that $\delta\varepsilon=\iota$. Since $\iota$ is the inclusion, and $\Soc(0:_{H_\fm^t({M})}x)=\Soc(0:_{H_\fm^t({M})}x^s)=\Soc H_{\fm}^t({M})$, we see that $\Soc \iota=1_{\Soc(H_{\fm}^t({M}))}$ and thus $(\Soc\delta)(\Soc\varepsilon)=1_{\Soc(H_{\fm}^t({M}))}$. It follows that $\Soc\delta=\Soc\delta^M_{x^s}$ is a split surjection.
\end{proof}

Given a system of parameters $x_1,...,x_d$ of $R$, the canonical map in \eqref{can_map2} to the direct limit $\smash{\xymatrix{R/(x_1^i,...,x_d^i) \ar[r] & \dirlim_j R/(x_1^j,...,x_d^j)}}$ is defined by the direct system
\[\xymatrix@C=4em{R/(x_1,...,x_d)\ar[r]^-{x_1\cdots x_d} & R/(x_1^2,...,x_d^2)\ar[r]^-{x_1\cdots x_d} & \cdots.
}\]
Moreover, there is a unique isomorphism, see \cite[Theorem 5.2.9]{BS13}:
\[\xymatrix{\dirlim_j R/(x_1^j,...,x_d^j)\ar[r]^-{\cong} & H_\fm^d(R).}\] 
\begin{lem}\label{mapsagreelemma}
Let $x_1,...,x_d$ be a system of parameters of $R$. The canonical map $R/(x_1,...,x_d) \to \dirlim_j R/(x_1^j,...,x_d^j)\cong H_\fm^d(R)$ agrees with the composition
\[
\xymatrix@C=1em{
H_\fm^0(R/(x_1,...,x_d))\ar[r]^-{} & H_\fm^1(R/(x_2,...,x_d)) \ar[r] & \cdots \ar[r] & H_\fm^{d-1}(R/(x_d)) \ar[r]^-{} & H_\fm^d(R)
}\]
of homomorphisms defined in Remark \ref{quotientconnectingmap}.
\end{lem}
\begin{proof}
First note that $R/(x_1,...,x_d)=H_\fm^0(R/(x_1,...,x_d))$. Further, the canonical map $R/(x_1,...,x_d) \to \dirlim_j R/(x_1^j,...,x_d^j)$ can be decomposed as the composition of the following canonical maps:
\[\xymatrix@C=1em{
\frac{R}{(x_1,...,x_d)} \ar[r] &  \dirlim_j \frac{R}{(x_1^j,x_2,...,x_d)} \ar[r] & \cdots \ar[r] &  \dirlim_j \frac{R}{(x_1^j,...,x_{d-1}^j,x_d)} \ar[r] & \dirlim_j \frac{R}{(x_1^j,...,x_d^j)}.
}\]
Fix $0< t \leq d$. 
It is therefore sufficient to show that the next two maps agree up to isomorphism; indeed, by \cite[Theorem 5.2.9]{BS13} there is a unique isomorphism between the domains, and another between the codomains:
\begin{align*}
&\xymatrix{\alpha_t:\dirlim_i \frac{R}{(x_1^i,...,x_{t-1}^i,x_{t},...,x_d)} \ar[r] & \dirlim_j\dirlim_{i} \frac{R}{(x_1^i,...,x_{t-1}^i,x_{t}^j,x_{t+1},...,x_d)},}\text{ and }\\
&\xymatrix{\delta_t:H_\fm^{t-1}(\frac{R}{(x_t,...,x_d)}) \ar[r] & H_\fm^t(\frac{R}{(x_{t+1},...,x_d)})},
\end{align*}
where $\alpha_t$ is the canonical map to the direct limit and $\delta_t=\delta^{R/(x_{t+1},...,x_d)}_{x_t}$ is the map defined in Remark \ref{quotientconnectingmap}.

For $0\leq u\leq d-1$, there is by \cite[Theorem 5.2.9]{BS13} a natural equivalence of functors $\dirlim_j( \frac{R}{(x_1^j,...,x_{u}^j)}\otimes_R-)\xrightarrow{\cong} H_{(x_1,...,x_u)}^u(-)$. This provides the isomorphisms and commutativity in the following diagram:
\[\xymatrix@C=1.7em{
\dirlim_{i}\frac{R}{(x_1^i,...,x_{t-1}^i,x_t,...,x_d)} \ar[r]^-{\cong} \ar[d]^{\alpha_t}& H_\fm^{t-1}(\frac{R}{(x_t,...,x_d)})\ar[d]^{\beta_t}\\
\dirlim_j\dirlim_{i}\frac{R}{(x_1^i,...,x_{t-1}^i,x_t^j,x_{t+1},...,x_d)} \ar[r]^-{\cong} & \dirlim_j H_\fm^{t-1}(\frac{R}{(x_t^j,x_{t+1},...,x_d)}),
}\]
where $\beta_t$ is the canonical map to the direct limit and the modules on the right have also utilized the fact that   $H_{(x_1,...,x_{t-1})}^{t-1}(\frac{R}{(x_t^j,x_{t+1},...,x_d)})\cong H_\fm^{t-1}(\frac{R}{(x_t^j,x_{t+1},...,x_d)})$ for each $j\geq 1$; see \cite[Exercise 2.1.9]{BS13}. It remains only to show that $\beta_t$ is isomorphic to $\delta_t$. To see this, consider the short exact sequence of direct systems:
\[\xymatrix{
0 \ar[r] & \frac{R/(x_{t+1},...,x_d)}{(0:x_t)} \ar[d] \ar[r]^{x_t} & R/(x_{t+1},...,x_d) \ar[r]\ar[d]^{x_t} & R/(x_t,...,x_d) \ar[r]\ar[d]^{x_t} & 0\\
0 \ar[r] & \frac{R/(x_{t+1},...,x_d)}{(0:x_t^2)} \ar[r]^{x_t^2}\ar[d] & R/(x_{t+1},...,x_d) \ar[r] \ar[d]^-{x_t}& R/(x_t^2,x_{t+1},...,x_d) \ar[r]\ar[d]^-{x_t} & 0\\
&\vdots&\vdots&\vdots&
}\]
Each row of this diagram yields a long exact sequence in local cohomology; along with the homomorphisms defined in Remark \ref{quotientconnectingmap}, this gives the following commutative diagram with exact rows: 
\[\xymatrix@C=1em{
H_\fm^{t-1}(\frac{R}{(x_{t+1},...,x_d)}) \ar[r]\ar[d]^{x_t} & H_\fm^{t-1}(\frac{R}{(x_t,...,x_d)}) \ar[r]^{\delta_t}\ar[d]^{x_t} & H_\fm^t(\frac{R}{(x_{t+1},...,x_d)}) \ar[r] \ar[d]^{=} & H_\fm^t(\frac{R}{(x_{t+1},...,x_d)})\ar[d]^{x_t}\\
H_\fm^{t-1}(\frac{R}{(x_{t+1},...,x_d)}) \ar[r]\ar[d]^-{x_t} & H_\fm^{t-1}(\frac{R}{(x_t^2,x_{t+1},...,x_d)}) \ar[r]\ar[d]^-{x_t} & H_\fm^t(\frac{R}{(x_{t+1},...,x_d)}) \ar[r]\ar[d]^{=} & H_\fm^t(\frac{R}{(x_{t+1},...,x_d)})\ar[d]^-{x_t}\\
\vdots&\vdots&\vdots&\vdots
}\]
The map from the second term in the first row, $H_\fm^{t-1}(\frac{R}{(x_t,...,x_d)})$, to the direct limit of the second column is $\beta_t$. The direct limits of the left and right columns are the localizations $(H_\fm^{t-1}(\frac{R}{(x_{t+1},...,x_d)}))_{x_t}$ and $(H_\fm^t(\frac{R}{(x_{t+1},...,x_d)}))_{x_t}$, respectively. These are zero since they are both $\fm$-torsion \cite[2.1.3]{BS13} and multiplication by $x_t\in \fm$ is invertible on either module. Hence the direct limit of the middle maps is an isomorphism, showing that $\beta_t$ and $\delta_t$ are isomorphic.
\end{proof}

In order to prove surjectivity of the socle of the map in \eqref{can_map2}, we introduce the following special types of systems of parameters. Let $M$ be a finitely generated $R$-module with $\dim M=t>0$.  As is standard, denote the annihilator of $H_\fm^i(M)$ by $\fa_i(M)=\ann_R H_\fm^i(M)$, and put $\fa(M)=\fa_0(M)\cdots \fa_{t-1}(M)$. A system of parameters $x_1,...,x_t$ of $M$ is called a \emph{$p$-standard system of parameters} if $x_t\in \fa(M)$ and $x_i\in \fa(M/(x_{i+1},...,x_t)M)$ for $i=1,...,t-1$. Such systems were defined at this level of generality by Cuong \cite{Cuo95}, who noted that a result of Schenzel \cite[Korollar 2.2.4]{Sch82} implies that every finitely generated $R$-module has a $p$-standard system of parameters provided $R$ has a dualizing complex. Indeed, if $R$ has a dualizing complex, then $\dim R/\fa(M)<t$ by \cite[Korollar 2.2.4]{Sch82} and so prime avoidance provides an element $x_t\in \fa(M)$ that is a parameter of $M$. Inductively, this shows the existence of $p$-standard systems of parameters, as well as the existence of the following variant\footnote{In fact, one may find by \cite[Korollar 2.2.4]{Sch82} an element in $\fa_{t-1}(M)$ that is a parameter of $M$, provided $R$ has a dualizing complex. Hence for our purposes, one may instead find a system of parameters $x_1,...,x_t$ of $M$ satisfying $x_i\in \fa_{i-1}(M/(x_{i+1}^s,...,x_{t}^s)M)$ for $i=1,...,t$. The notion considered in Definition \ref{ps_standard} is chosen to be reminiscent of $p$-standard systems.}, provided $R$ has a dualizing complex.

\begin{dfn}\label{ps_standard}
Let $M$ be a finitely generated $R$-module with $\dim M=t$. For an integer $s\geq 1$, a system of parameters $x_1,...,x_t$ of $M$ is called a \emph{$p_s$-standard system of parameters} if $x_t\in \fa(M)$ and $x_i\in \fa(M/(x_{i+1}^s,...,x_t^s)M)$ for $i=1,...,t-1$. 
\end{dfn}
\noindent
Evidently, $p_1$-standard and $p$-standard systems of parameters are the same, but the relationship between $p_s$-standard and $p$-standard systems of parameters is less straightforward for $s\geq 2$. 

\begin{rmk}\label{CuongCuong}
If $x_1,...,x_t$ is a $p_s$-standard system of parameters of $M$ for an integer $s\geq 1$, then $x_1^s,...,x_t^s$ is a $p$-standard system of parameters of $M$; this follows straight from the definitions. Thus the existence of a $p_s$-standard system of parameters of $M$ for some $s\geq 1$ implies the existence of a $p$-standard system of parameters of $M$. 

Conversely, it follows from work of Cuong and Cuong \cite{CC07} that the existence of a $p$-standard system of parameters of $M$ implies the existence of a $p_s$-standard system of parameters of $M$ for every $s\geq 1$: indeed, if $x_1,...,x_t$ is a $p$-standard system of parameters of $M$ and $s\geq 1$, then  \cite[Corollary 3.9]{CC07} implies that $x_1^t,...,x_i^t,x_{i+1}^{st},...,x_t^{st}$ is also $p$-standard for each $i=1,...,t$ hence $x_1^t,...,x_t^t$ is a $p_s$-standard system of parameters of $M$. Moreover, Cuong and Cuong prove in \cite[Theorem 5.2]{CC17} that the existence of a $p$-standard system of parameters is equivalent to the ring $R$ being a quotient of a local Cohen-Macaulay ring.

In this paper, we consistently express our results in terms of $p_s$-standard systems of parameters so that the exposition is relatively self-contained, but combining with the aforementioned results one can reformulate our results in terms of $p$-standard systems of parameters and observe that Theorem \ref{Gor_char_higherdim} below applies to any ring $R$ that is a quotient of a local Cohen-Macaulay ring.
\end{rmk}

We now come to proving surjectivity of the socle of the map in \eqref{can_map2} for $p_s$-standard systems of parameters if $s\geq 2$. The main distinction between the next result and \cite[Proposition 2.5]{MRS08} or \cite[Lemma 3.12]{GS03} is that here we have some control for the point at which the induced maps on socles are surjective.
\begin{prp}\label{simple_surjection}
Let $x_1,...,x_d$ be a system of parameters of $R$. If $x_1,...,x_d$ is a $p_s$-standard system of parameters for some $s\geq 2$, then the canonical map $R/(x_1^s,...,x_d^s)\to \dirlim_j R/(x_1^j,...,x_d^j)\cong H_\fm^d(R)$ induces a split surjection
\[\xymatrix{
\Soc R/(x_1^s,...,x_d^s)\ar[r] & \Soc H_{\fm}^d(R).
}\]
Hence the canonical map $\Soc R/(x_1^i,...,x_d^i)\to \Soc H_\fm^d(R)$ is surjective for $i\geq s$.
\end{prp}
\begin{proof}
Suppose $x_1,...,x_d$ is a $p_s$-standard system of parameters of $R$ for some $s\geq 2$. By definition, we have $x_t\in \fa(R/(x_{t+1}^s,...,x_d^s)) \subseteq \ann_RH_\fm^{t-1}(R/(x_{t+1}^s,...,x_d^s))$ for each $t=1,...,d$, and so Lemma \ref{topcase} yields that the induced map
\[\xymatrix{
\Soc H_\fm^{t-1}({R}/{(x_t^s,x_{t+1}^s,...,x_d^s)}) \ar[r] & \Soc H_\fm^t({R}/{(x_{t+1}^s,...,x_d^s)})
}\]
is a split surjection for each $t=1,...,d$.  Thus the composition 
\[\xymatrix@C=1em{
H_{\fm}^{0}(R/(x_{1}^s,...,x_d^s))\ar[r] & H_{\fm}^{1}(R/(x_{2}^s,...,x_d^s)) \ar[r] &   \cdots \ar[r] & H_\fm^{d-1}({R}/{(x_d^s)}) \ar[r] & H_\fm^d(R),
}\]
which agrees with the canonical map $R/(x_1^s,...,x_d^s) \to \dirlim_j R/(x_1^j,...,x_d^j)\cong H_\fm^d(R)$ by Lemma \ref{mapsagreelemma}, induces the desired split surjection on socles.

The final remark follows from the definition of the direct system.
\end{proof}

\begin{rmk}\label{detmaps}
Let $y_1,...,y_d$ and $x_1,...,x_d$ be systems of parameters of $R$ such that $(y_1,...,y_d)\subseteq (x_1,...,x_d)$. Suppose $A=(a_{ij})$ and $B=(b_{ij})$ are matrices such that $y_i=\sum_{j=1}^d a_{ij}x_j=\sum_{j=1}^d b_{ij}x_j$. One has $\det A\cdot (x_1,...,x_d)\subseteq (y_1,...,y_d)$, thus multiplication by $\det A$ induces a well-defined map $R/(x_1,...,x_d)\xrightarrow{\det A}R/(y_1,...,y_d)$; see \cite[p. 2681]{FH11}. Moreover, the proof of \cite[Corollary 2.5]{FH11} shows that multiplication by either $\det A$ or $\det B$ determines the same map from $R/(x_1,...,x_d)$ to $R/(y_1,...,y_d)$.
\end{rmk}

After a reduction in order to apply Proposition \ref{simple_surjection} in place of \cite[Proposition 2.5]{MRS08}, the proof of the next result is similar to that of \cite[Theorem 2.7]{MRS08}. Let $s\geq 1$ be an integer and recall from above that a $p_s$-standard system of parameters exists if $R$ has a dualizing complex, for example if $R$ is complete; see also Remark \ref{CuongCuong}.
\begin{thm}\label{Gor_char_higherdim}
Suppose there exists a $p_s$-standard system of parameters $x_1,...,x_d$ of $R$ for some $s\geq 2$. The ring $R$ is Gorenstein if and only if some parameter ideal contained in $(x_1^s,...,x_d^s)$ is irreducible.
\end{thm}
\begin{proof}
Suppose $x_1,...,x_d$ is a $p_s$-standard system of parameters of $R$ for some $s\geq 2$. All parameter ideals in a Gorenstein ring are irreducible \cite[Theorem 18.1]{Mat89}, so it is sufficient to prove the converse. 

Assume $y_1,...,y_d$ is a system of parameters such that $(y_1,...,y_d)$ is irreducible and contained in $(x_1^s,...,x_d^s)$. We first claim that, for the direct system
\[\xymatrix@C=5em{
R/(y_1,...,y_d) \ar[r]^{y_1\cdots y_d} & R/(y_1^2,...,y_d^2) \ar[r]^-{y_1\cdots y_d} & \cdots,
}\]
the canonical map $R/(y_1,...,y_d)\to \dirlim_j R/(y_1^j,...,y_d^j)\cong H_\fm^d(R)$ is surjective when restricted to socles. As $x_1^j,...,x_d^j$ and $y_1^j,...,y_d^j$ are systems of parameters for all $j\geq 1$, there exist families of positive integers $\{1=t_1<t_2<\cdots\}$, $\{u_1<u_2<\cdots\}$, and $\{v_1<v_2<\cdots\}$ such that for each $i\geq 1$ we have containments:
$$(x_1^{is},...,x_d^{is})\supseteq (y_1^{t_i},...,y_d^{t_i})\supseteq (x_1^{u_i},...,x_d^{u_i})\supseteq (y_1^{v_i},...,y_d^{v_i}).$$
By Remark \ref{detmaps}, we obtain maps (coming from determinants of matrices) making the following commutative diagram of direct systems:
\[\xymatrix{
R/(x_1^{s},...,x_d^{s}) \ar[r]^{\sigma_s} \ar[d]_{x_1^s\cdots x_d^s} & R/(y_1^{},...,y_d^{}) \ar[r]^{\tau_s} \ar[d]_{y_1^{t_2-1}\cdots y_d^{t_2-1}} & R/(x_1^{u_1},...,x_d^{u_1}) \ar[r]^{\rho_s} \ar[d]_{x_1^{u_2-u_1}\cdots x_d^{u_2-u_1}} & R/(y_1^{v_1},...,y_d^{v_1})\ar[d]_{y_1^{v_2-v_1}\cdots y_d^{v_2-v_1}}\\
R/(x_1^{2s},...,x_d^{2s}) \ar[r]^{\sigma_{2s}} \ar[d]_-{x_1^s\cdots x_d^s} & R/(y_1^{t_2},...,y_d^{t_2}) \ar[r]^{\tau_{2s}} \ar[d]_-{y_1^{t_3-t_2}\cdots y_d^{t_3-t_2}} & R/(x_1^{u_2},...,x_d^{u_2}) \ar[r]^{\rho_{2s}} \ar[d]_-{x_1^{u_3-u_2}\cdots x_d^{u_3-u_2}} & R/(y_1^{v_2},...,y_d^{v_2})\ar[d]_-{y_1^{v_3-v_2}\cdots y_d^{v_3-v_2}}\\
\vdots & \vdots & \vdots & \vdots
}\]
Moreover, Remark \ref{detmaps} also yields that the compositions of horizontal maps are the familiar ones: $\tau_{is}\sigma_{is}=x_1^{u_i-is}\cdots x_d^{u_i-is}$ and $\rho_{is}\tau_{is}=y_1^{v_i-t_i}\cdots y_d^{v_i-t_i}$, for $i\geq 1$.  The direct limits of all four columns are isomorphic to $H_\fm^d(R)$, and it thus follows from the universal property of direct limits that the induced maps on these direct limits are isomorphisms. Hence we obtain a commutative diagram of canonical maps:
\[\xymatrix@C=4em{
\Soc R/(x_1^s,...,x_d^s) \ar[r]^{\Soc\sigma_s}\ar[d] & \Soc R/(y_1,...,y_d)\ar[d]\\
\Soc H_\fm^d(R) \ar[r]^{\cong} & \Soc H_\fm^d(R)
}\]
The left vertical map is a surjection by Proposition \ref{simple_surjection}, hence it follows that the right vertical map is a surjection as well.

Let $\phi:R/(y_1,...,y_d)\to \dirlim_j R/(y_1^j,...,y_d^j)\cong H_\fm^d(R)$ be the canonical map and proceed as in the proof of \cite[Theorem 2.7]{MRS08}: Recall that the limit closure of $y_1,...,y_d$ is defined as $\{y_1,...,y_d\}_R^{\operatorname{lim}}=\bigcup_{n\geq 0}((y_1^{n+1},...,y_d^{n+1}):_Ry_1^n\cdots y_d^n)$. By \cite[Remark 2.2]{MRS08}, $\ker(\phi)=\{y_1,...,y_d\}_R^{\operatorname{lim}}/(y_1,...,y_d)$.  Applying $\Soc(-)$ to the canonical maps, we obtain the next exact sequence, where surjectivity of the map on the right was shown above:
\[
\xymatrix{
0\ar[r] & \Soc\ker(\phi) \ar[r] & \Soc R/(y_1,...,y_d) \ar[r] & \Soc H_\fm^d(R) \ar[r] & 0.
}\]
Irreducibility of $(y_1,...,y_d)$ yields $\dim_{R/\fm}\Soc R/(y_1,...,y_d)=1$. As $H_\fm^d(R)\not=0$ we obtain $\Soc \ker(\phi)=0$ implying per \cite[Proposition 2.3]{MRS08} that $y_1,...,y_d$ is regular. Thus $R$ is Cohen-Macaulay with $\dim_{R/\fm} \Soc H_\fm^d(R)=1$, hence Gorenstein.
\end{proof}

As an immediate consequence of Theorem \ref{Gor_char_higherdim}, we obtain Theorem \ref{intro2} from the introduction.

\begin{rmk}
If $R$ has finite local cohomologies, in which case there is an integer $n_0>0$ such that $\fm^{n_0} H_\fm^i(R)=0$ for $i<d$ (this is not assumed in the results above), a similar bound can be obtained by using \cite[Corollary 4.3]{CQ11} in conjunction with \cite[Theorem 2.7]{MRS08}; in particular, it follows from these that $R$ is Gorenstein if and only if some parameter ideal contained in $\fm^{2n_0}$ is irreducible. In particular, if $R$ is Buchsbaum so that we have $\fm H_\fm^i(R)=0$ for $i<d$, then $R$ is Gorenstein if and only if some parameter ideal contained in $\fm^2$ is irreducible.
\end{rmk}

\section*{Acknowledgments}
The authors are grateful to Thomas Marley for comments on an early version of this paper and to the anonymous referee for helpful suggestions.


\end{document}